\documentclass[12pt,a4paper]{article}
\usepackage{amsmath}
\usepackage{amsfonts}
\usepackage{amssymb}
\usepackage{graphicx}
\usepackage{tikz}
\usetikzlibrary{decorations.pathmorphing} 
\usepackage{amssymb}
\usepackage{amsthm} 
\usepackage{gensymb} 
\usepackage{hyperref}
\usepackage{xcolor} 
\hypersetup{colorlinks, linkcolor={blue}, citecolor={green!80!black}, urlcolor={blue}}
\usepackage{algorithm}
\usepackage[noend]{algpseudocode} 
\usepackage{comment}
\usepackage{pdfpages} 
\usepackage{pdflscape}
\usepackage{caption}
\usepackage{enumitem} 
\usepackage{authblk}
\usepackage{multirow} 
\usepackage{tabularx}
\usepackage{arydshln} 

\newtheorem{theorem}{Theorem}[section]
\newtheorem{lemma}[theorem]{Lemma}
\newtheorem{proposition}[theorem]{Proposition}

\newtheorem{conjecture}{Conjecture}[section]
\newtheorem{corollary}[theorem]{Corollary}

\makeatletter
\renewcommand{\make@df@tag@@@}[2][]{%
	\gdef\df@tag{%
		\tagform@{#2\rlap{\hphantom)#1}}%
		\toks@\@xp{\p@equation{#2}}%
		\edef\@currentlabel{\the\toks@}%
	}%
}
\makeatother

\newcommand{\abs}[1]{\lvert#1\rvert}

\allowdisplaybreaks 

\title{Factorisation of Greedoid Polynomials of Rooted Digraphs}

\author[1]{Kai Siong Yow\thanks{Email: \texttt{KaiSiong.Yow@gmail.com}}}
\author[2]{Kerri Morgan\thanks{Email: \texttt{Kerri.Morgan@deakin.edu.au}}}
\author[1]{Graham Farr\thanks{Email: \texttt{Graham.Farr@monash.edu}}}
\affil[1]{Faculty of Information Technology, Monash University, Clayton, Victoria 3800, Australia}
\affil[2]{Deakin University, Geelong, Australia, School of Information Technology, Faculty of Science Engineering \& Built Environment}

\date{\today}     

\begin{document}

\maketitle

\begin{abstract}
	Gordon and McMahon defined a two-variable greedoid polynomial $ f(G;t,z) $ for any greedoid $ G $. They studied greedoid polynomials for greedoids associated with rooted graphs and rooted digraphs. They proved that greedoid polynomials of rooted digraphs have the multiplicative direct sum property. In addition, these polynomials are divisible by $ 1 + z $ under certain conditions. We compute the greedoid polynomials for all rooted digraphs up to order six. A greedoid polynomial $ f(D) $ of a rooted digraph $ D $ of order $ n $ \emph{GM-factorises} if $ f(D) = f(G) \cdot f(H) $ such that $ G $ and $ H $ are rooted digraphs of order at most $ n $ and $ f(G),f(H) \ne 1 $. We study the GM-factorability of greedoid polynomials of rooted digraphs, particularly those that are not divisible by $ 1 + z $. We give some examples and an infinite family of rooted digraphs that are not direct sums but their greedoid polynomials GM-factorise.
\end{abstract}


\section{Introduction}

Greedoids were introduced by Korte and Lov\'{a}sz as collections of sets that generalise matroids~\cite{KorteLovasz1981}. Korte and Lov\'{a}sz observed that the optimality of some ``greedy" algorithms including breadth-first search could be traced back to an underlying combinatorial structure that satisfies the greedoid, but not the matroid, framework. Bj{\"o}rner and Ziegler~\cite{BjornerZ1992} used two algorithmic constructions of a minimum spanning tree of a connected graph, i.e., Kruskal's and Prim's algorithms, to distinguish between greedoids and matroids. For each step in both algorithms, an edge with the minimum weight is added into the minimum spanning tree. The edge sets of the trees/forests that are obtained in each step form the feasible sets of a greedoid. Feasible sets obtained via Kruskal's algorithm remain feasible when removing any edge from the sets. However, this is not always true for feasible sets that are obtained via Prim's algorithm. Therefore, the greedoid that is obtained by using Kruskal's algorithm (but not Prim's algorithm) is in fact a matroid.

There are two equivalent ways to define greedoids, using set systems or hereditary languages~\cite{KorteLovasz1985,KorteLS1991}. We define greedoids based on set systems. A \emph{greedoid} over a finite ground set $ E $ is a pair $ (E,F) $ where $ F \subseteq 2^{E} $ is a collection of subsets of $ E $ (called the \emph{feasible sets}) satisfying:
\begin{itemize}\label{item:greedoid}
	\item[(G1)] For every non-empty $ X \in F $, there is an element $ x \in X $ such that $ X-\{x\} \in F $.
	\item[(G2)] For $ X,Y \in F $ with $ \abs{X} < \abs{Y} $, there is an element $ y \in Y-X $ such that $ X \cup \{y\} \in F $.
\end{itemize}

The \emph{rank} $ r(A) $ of a subset $ A \subseteq E $ in a greedoid $ (E,F) $ is defined as $ r(A)=\hbox{max}\{\abs{X}:X \subseteq A, X \in F\} $. Any greedoid is uniquely determined by its rank function.
\begin{theorem}\emph{\cite{KorteLovasz1984}}
	A function $ r:2^{E} \mapsto \mathbb{N} \cup \{0\} $ is the rank function of a greedoid $ (E,F) $ if and only if for all $ X,Y \subseteq E $ and for all $ x,y \in E $ the following conditions hold:
	\begin{itemize}
		\item[\emph{(R1)}] \emph{$r(X) \le \abs{X} $.}
		\item[\emph{(R2)}] \emph{If $ X \subseteq Y $, then $ r(X) \le r(Y) $.}
		\item[\emph{(R3)}] \emph{If $ r(X)=r(X \cup \{x\})=r(X \cup \{y\}) $, then $ r(X)=r(X \cup \{x\} \cup \{y\}) $.}
	\end{itemize}
\end{theorem}

Important classes of greedoids are those associated with rooted graphs and rooted digraphs. These are called \emph{branching greedoids} and \emph{directed branching greedoids}, respectively. We focus on directed branching greedoids. Hence, all our digraphs are rooted.

An \emph{arborescence}~\cite{Tutte1973} is a directed tree rooted at a vertex $ v $ such that every edge that is incident with $ v $ is an outgoing edge, and exactly one edge is directed into each of the other vertices. For every non-root vertex in an arborescence, there exists a unique directed path in the arborescence that leads from the root vertex to the non-root vertex. Occasionally, to highlight this property, people describe the root vertex as \emph{Rome}\footnote{From the proverb: All roads lead to Rome.}~\cite{Tutte1973}. Some authors define arborescences by reversing the direction of each edge in our definition, giving a set of arborescences that is different to ours. In this scenario, each unique directed path in the arborescence directs into rather than away from the root vertex. In both definitions, the number of arborescences rooted at each vertex is identical. To change from one definition to the other, simply reverse the direction for all the edges.

Let $ D $ be a rooted digraph. A subdigraph $ F $ of $ D $ is \emph{feasible} if $ F $ is an arborescence. We call the edge set of $ F $ a \emph{feasbile set}. If the edge set of $ F $ is maximal, then it is a \emph{basis}. A \emph{spanning arborescence} of $ D $ is a subdigraph of $ D $ that is an arborescence which includes every vertex of $ D $. The \emph{rank} of a subset $ X \subseteq E(D) $ is defined as $ r(X)=\hbox{max}\{\abs{A}:A \subseteq X, A \text{ is feasible}\} $.

A \emph{directed branching greedoid} over a finite set $ E $ of directed edges of a rooted digraph is a pair $ (E,F) $ where $ F $ is the set of feasible subsets of $ E $. This was defined and shown to be a greedoid by Korte and Lov\'{a}sz~\cite{KorteLovasz1984}.

Let $ G $ be a greedoid. Gordon and McMahon~\cite{GordonMcMahon1989} defined a two-variable greedoid polynomial of $ G $
\begin{equation*}
	f(G;t,z) = \sum_{A \subseteq E(G)} t^{r(G)-r(A)} z^{\abs{A}-r(A)}
\end{equation*}
which generalises the one-variable greedoid polynomial $ \lambda(G;t) $ given by Bj{\"o}rner and Ziegler in~\cite{BjornerZ1992}. We call the two-variable greedoid polynomial $ f(G;t,z) $ the \emph{greedoid polynomial}. The greedoid polynomial is motivated by the Tutte polynomial of a matroid~\cite{Tutte1954}, and is an analogue of the Whitney rank generating function~\cite{Whitney1932}. This polynomial is one of the digraph polynomials that is analogous of the Tutte polynomial. A survey of such polynomials for directed graphs can be found in~\cite{Chow2018}.

Gordon and McMahon studied greedoid polynomials for branching greedoids and directed branching greedoids. They showed that $ f(D;t,z) $ can be used to determine if a rooted digraph $ D $ is a rooted arborescence~\cite{GordonMcMahon1989}. However, this result does not hold when $ D $ is an unrooted tree~\cite{EisenstatGordon2006}.

Suppose $ D $, $ D_{1} $ and $ D_{2} $ are rooted digraphs, and $ E(D_{1}), E(D_{2}) \subseteq E(D) $. The digraph $ D $ is the \emph{direct sum} of $ D_{1} $ and $ D_{2} $, if $ E(D_{1}) \cup E(D_{2}) = E(D) $, $ E(D_{1}) \cap E(D_{2}) = \emptyset $ and the feasible sets of $ D $ are precisely the unions of feasible sets of $ D_{1} $ and $ D_{2} $. Gordon and McMahon proved that the greedoid polynomials of rooted digraphs have the multiplicative direct sum property, that is, if $ D $ is the direct sum of $ D_{1} $ and $ D_{2} $, then $ f(D;t,z) = f(D_{1};t,z) \cdot f(D_{2};t,z) $. This raises the question of whether this is the only circumstance in which this polynomial can be factorised. The Tutte polynomial of a graph $ G $ factorises if and only if $ G $ is a direct sum~\cite{MerinodeMierNoy2001}, but the situation for the chromatic polynomial is more complex~\cite{MorganF2009}.

Gordon and McMahon showed that the greedoid polynomial of a rooted digraph that is not necessarily a direct sum has $ 1 + z $ among its factors under certain conditions (see Theorems~\ref{thm:McMahon1993} and \ref{thm:GordonMcMahon1997}). We address more general types of factorisation in this article.

Gordon and McMahon gave a recurrence formula to compute $ f(D;t,z) $ where $ D $ is a rooted digraph. The following proposition gives the formula, which involves the usual deletion-contraction operations.
\begin{proposition}\emph{\cite{GordonMcMahon1989}}\label{pro:recurrence}
	Let $ D $ be a digraph rooted at a vertex $ v $, and $ e $ be an outgoing edge of $ v $. Then
	\begin{equation*}
		f(D;t,z) = f(D/e;t,z) + t^{r(D)-r(D\setminus e)}f(D\setminus e;t,z).
	\end{equation*}
\end{proposition}

A \emph{greedoid loop}~\cite{McMahon1993} in a rooted graph, or a rooted digraph, is an edge that is in no feasible set. It is either an ordinary (directed) loop, or an edge that belongs to no (directed) path from the root vertex.

\begin{theorem}\emph{\cite{McMahon1993}}\label{thm:McMahon1993}
	Let $ D $ be a rooted digraph that has no greedoid loops. Then $ D $ has a directed cycle if and only if $ 1+z $ divides $ f(D) $.
\end{theorem}

Let $ G $ be a greedoid. A subset $ S \subseteq E(G) $ is \emph{spanning} if $ S $ contains a basis. Gordon and McMahon gave a graph-theoretic interpretation for the highest power of $ 1 + z $ which divides $ f(G) $ in the following theorem.
\begin{theorem}\emph{\cite{GordonMcMahon1997}}\label{thm:GordonMcMahon1997}
	Let $ G $ be the directed branching greedoid associated with a rooted digraph $ D $ with no greedoid loops or isolated vertices. If $ f(G;t,z) = (1 + z)^{k}h(t,z) $, where $ 1 + z $ does not divide $ h(t,z) $, then $ k $ is the minimum number of edges that need to be removed from $ D $ to leave a spanning acyclic directed graph.
\end{theorem}

Tedford~\cite{Tedford2009} defined a three-variable greedoid polynomial $ f(G;t,p,q) $ for any finite rooted graph $ G $, which generalises the two-variable greedoid polynomial. He showed that $ f(G;t,p,q) $ obeys a recursive formula. He also proved that $ f(G;t,p,q) $ determines the number of greedoid loops in any rooted graph $ G $. His main result shows that $ f(G;t,p,q) $ distinguishes connected rooted graphs $ G $ that are loopless and have at most one cycle. He extended $ f(G;t,p,q) $ from rooted graphs to general greedoids, and proved that the polynomial determines the number of loops for a larger class of greedoids.

In this article, we compute the greedoid polynomials for all rooted digraphs (up to isomorphism unless otherwise stated) up to order six. All the labelled rooted digraphs (without loops and multiple edges, but cycles of size two are allowed) up to order six were provided by Brendan McKay\footnote{More combinatorial data can be found at \href{https://users.cecs.anu.edu.au/~bdm/data/}{https://users.cecs.anu.edu.au/\textasciitilde bdm/data/}.} on 28 March 2018 (personal communication from McKay to Farr). We then study the factorability of these polynomials, particularly those that are not divisible by $ 1 + z $.

Two rooted digraphs are \emph{GM-equivalent} if they both have the same greedoid polynomial. If a rooted digraph is a direct sum, then it is \emph{separable}. Otherwise, it is \emph{non-separable}.

A greedoid polynomial $ f(D) $ of a rooted digraph $ D $ of order $ n $ \emph{GM-factorises} if $ f(D) = f(G) \cdot f(H) $ such that $ G $ and $ H $ are rooted digraphs of order at most $ n $ and $ f(G),f(H) \ne 1 $. Note that $ f(G) $ and $ f(H) $ are not necessarily distinct. The polynomials $ f(G) $ and $ f(H) $ are \emph{GM-factors} of $ f(D) $. We also say a rooted digraph $ D $ \emph{GM-factorises} if its greedoid polynomial GM-factorises. Every rooted digraph that is a direct sum has a GM-factorisation.

An irreducible GM-factor is \emph{basic} if the GM-factor is either $ 1 + t $ or $ 1 + z $. Otherwise, the irreducible GM-factor is \emph{nonbasic}. We are most interested in nonbasic GM-factors. A GM-factor is \emph{primary} if it is irreducible, nonbasic and is not a GM-factor of any greedoid polynomial of rooted digraphs of smaller order. Such a factor appears as a GM-factor only in rooted digraphs with at least as many vertices as the current one. For $ k \ge 1 $, a non-separable digraph is \emph{$ k $-nonbasic} if its greedoid polynomial has $ k $ nonbasic GM-factors. A non-separable digraph is \emph{totally $ k $-nonbasic} if it is $ k $-nonbasic and contains no basic GM-factors. Likewise, a non-separable digraph is \emph{$ k $-primary} if its greedoid polynomial has $ k $ primary GM-factors. A non-separable digraph is \emph{totally $ k $-primary} if it is $ k $-primary and contains no basic GM-factors. It follows that if a non-separable digraph is (totally) $ k $-primary, then the digraph is (totally) $ k $-nonbasic.

Our results show that there exist non-separable digraphs that GM-factorise and their polynomials have neither $ 1 + t $ nor $ 1 + z $ as factors. In some cases (but not all), these non-separable digraphs of order $ n $ are GM-equivalent to a separable digraph of order at most $ n $. We give the numbers of polynomials of this type of non-separable digraph. For rooted digraphs up to order six and $ k \ge 2 $, we found that there exist no $ (k+1) $-nonbasic digraphs and no $ k $-primary digraphs. We also provide the numbers of $ 2 $-nonbasic digraphs, totally $ 2 $-nonbasic digraphs, $ 1 $-primary digraphs and totally $ 1 $-primary digraphs. We then give the first examples of totally $ 2 $-nonbasic and totally $ 1 $-primary digraphs. Lastly, we give an infinite family of non-separable digraphs where their greedoid polynomials factorise into at least two non-basic GM-factors.

\section{Results}

The greedoid polynomials of all rooted digraphs up to order six were computed based on the deletion-contraction recurrence in Proposition~\ref{pro:recurrence}. We simplified and factorised all these greedoid polynomials using Wolfram Mathematica.

Throughout, rooted digraphs are up to isomorphism unless stated otherwise.

\subsection{Separability and Non-separability}

For each order, we determined the numbers of rooted digraphs, separable digraphs, non-separable digraphs, and non-separable digraphs of order $ n $ that are GM-equivalent to some separable digraph of order at most $ n $ (see Table~\ref{table:summary}, and the list of abbreviations in Table~\ref{table:abbreviation_table_1}).

Note that the sequences of numbers of labelled rooted digraphs (T) and rooted digraphs (T-ISO) are not listed in The On-Line Encyclopedia of Integer Sequences (OEIS).

We observe that the ratio of T-ISO to T shows an increasing trend. The ratio of NS to T-ISO is also increasing (for $ n \ge 3 $), as expected.

\begin{table}[H]
	\centering
	\begin{tabularx}{\linewidth}{|l|X|}
		\hline
		\multicolumn{1}{|c}{Abbreviation} & \multicolumn{1}{|c|}{Description}\\
		\hline \hline
		
		T & Number of labelled rooted digraphs\\
		\hline
		T-ISO & Number of rooted digraphs\\
		\hline
		S & Number of separable digraphs\\
		\hline
		NS & Number of non-separable digraphs\\
		\hline
		NSE & Number of non-separable digraphs of order $ n $ that are GM-equivalent to some separable digraph of order at most $ n $\\
		
		\hline
	\end{tabularx}
	\caption{Abbreviations for Table~\ref{table:summary}}
	\label{table:abbreviation_table_1}
\end{table}

\begin{table}[H]
	\centering
	\begin{tabular}{|c|r||r|r|r|r|}
		\hline
		\multicolumn{1}{|c}{$ n $} & \multicolumn{1}{|c||}{T} & \multicolumn{1}{c|}{T-ISO} & \multicolumn{1}{c|}{S} & \multicolumn{1}{c|}{NS} & \multicolumn{1}{c|}{NSE}\\
		\hline \hline
		
		1 & 1 & 1 & 0 & 1 & 0\\
		
		2 & 6 & 4 & 0 & 4 & 0\\
		
		3 & 48 & 36 & 6 & 30 & 7\\
		
		4 & 872 & 752 & 88 & 664 & 200\\
		
		5 & 48040 & 45960 & 2404 & 43556 & 10641\\
		
		6 & 9245664 & 9133760 & 150066 & 8983694 & 1453437\\		
		\hline
	\end{tabular}
	\caption{Numbers of various types of rooted digraphs (up to order six)}
	\label{table:summary}
\end{table}

For each order, we also provide the number PU of unique greedoid polynomials and the ratio of PU to T-ISO, in Table~\ref{table:unique_poly}.
\begin{table}[H]
	\centering
	\begin{tabular}{|c|r|r|r|}
		\hline
		\multicolumn{1}{|c}{$ n $} & \multicolumn{1}{|c|}{T-ISO} & \multicolumn{1}{c|}{PU} & \multicolumn{1}{c|}{PU/T-ISO}\\
		\hline \hline
		
		1 & 1 & 1 & 1.0000\\
		
		2 & 4 & 4 & 1.0000\\
		
		3 & 36 & 22 & 0.6111\\
		
		4 & 752 & 201 & 0.2673\\
		
		5 & 45960 & 6136 & 0.1335\\
		
		6 & 9133760 & 849430 & 0.0930\\		
		\hline
	\end{tabular}
	\caption{Numbers PU of unique greedoid polynomials of rooted digraphs (up to order six) and the ratio of PU to T-ISO}
	\label{table:unique_poly}
\end{table}
\noindent
Bollob\'{a}s, Pebody and Riordan conjectured that almost all graphs are determined by their chromatic or Tutte polynomials~\cite{BollobasPR2000}. However, this conjecture does not hold for matroids. The ratio of the number of unique Tutte polynomials of matroids to the number of non-isomorphic matroids approaches 0 as the cardinality of matroids increases, which can be shown using the bounds given in Exercise 6.9 in~\cite{BrylawskiO1992}. We believe that greedoid polynomials of rooted digraphs behave in a similar manner as matroids. According to our findings, the ratio of PU to T-ISO shows a decreasing trend. We expect that as $ n $ increases, this ratio continues to decrease. The question is, does this ratio eventually approach 0 or is it bounded away from 0? Further computation may give more insight on this question.

\subsection{Factorability}

For $ n \in \{1, \ldots, 5\} $, we identified the numbers of greedoid polynomials that GM-factorise for rooted digraphs of order $ n $. Details are given in Table~\ref{table:factorability} (see Table~\ref{table:abbreviation_table_2} for the list of abbreviations and Figure~\ref{fig:VennDiagram_factorability} for the corresponding Venn diagram).
\begin{table}[H]
	\centering
	\begin{tabularx}{\linewidth}{|l|X|}
		\hline
		\multicolumn{1}{|c}{Abbreviation} & \multicolumn{1}{|c|}{Description}\\
		\hline \hline
		
		& Number of unique greedoid polynomials of rooted digraphs that \ldots\\
		\hline
		PNF & \ldots cannot be GM-factorised\\
		\hline
		PF & \ldots can be GM-factorised\\
		\hline
		PFS & \ldots can be GM-factorised and the digraph is separable\\
		\hline
		PFNS & \ldots can be GM-factorised and the digraph is non-separable\\
		\hline
		\hline
		PF & PFS $ \cup $ PFNS\\
		\hline
		COMM & PFS $ \cap $ PFNS\\
		\hline
		PFSU & PFS $ - $ COMM\\
		\hline
		PFNSU & PFNS $ - $ COMM\\
		\hline
		
		\hline
	\end{tabularx}
	\caption{Abbreviations for Figure~\ref{fig:VennDiagram_factorability} and Table~\ref{table:factorability}}
	\label{table:abbreviation_table_2}
\end{table}

\begin{figure}[H]
	\centering
	\includegraphics[scale=1]{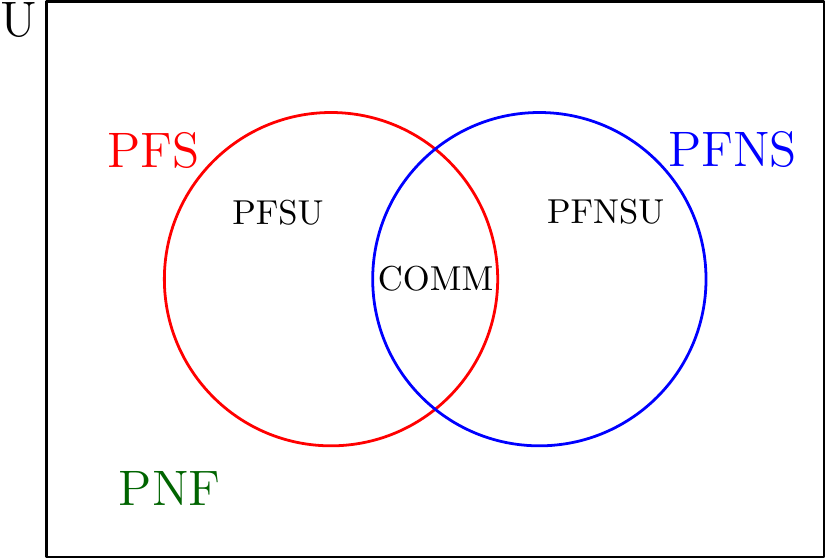}
	\caption{Venn diagram that represents the factorability of greedoid polynomials of rooted digraphs where U $ = $ PF $ \cup $ PNF and PF $ = $ PFS $ \cup $ PFNS}
	\label{fig:VennDiagram_factorability}
\end{figure}

\begin{table}[H]
	\centering
	\begin{tabular}{|c|r|r|r|r|r|r|r|}
		\hline
		\multicolumn{1}{|c}{$ n $} & \multicolumn{1}{|c}{PNF} & \multicolumn{1}{|c}{PF} & \multicolumn{1}{|c}{PFS} & \multicolumn{1}{|c}{PFNS} & \multicolumn{1}{|c}{COMM} & \multicolumn{1}{|c}{PFSU} & \multicolumn{1}{|c|}{PFNSU}\\
		\hline \hline
		
		1 & 1 & 0 & 0 & 0 & 0 & 0 & 0\\
		
		2 & 3 & 1 & 0 & 1 & 0 & 0 & 1\\
		
		3 & 6 & 16 & 6 & 13 & 3 & 3 & 10\\
		
		4 & 37 & 164 & 41 & 145 & 22 & 19 & 123\\
		
		5 & 1044 & 5092 & 444 & 4867 & 219 & 225 & 4648\\		
		\hline
	\end{tabular}
	\caption{Factorability of greedoid polynomials of rooted digraphs (up to order five)}
	\label{table:factorability}
\end{table}

We found that the ratio of PF to PU shows an upward trend, and the ratio stands at 0.8299 when $ n=5 $. It seems that most likely as $ n $ increases, the ratio will either approach 1 in which case almost all greedoid polynomials of rooted digraphs GM-factorise, or the ratio will approach a fixed $ \alpha $ where $ 0.8299 \le \alpha < 1 $. We ask, what is the limiting proportion of greedoid polynomials of rooted digraphs that GM-factorise, as $ n \rightarrow \infty $?

We categorised these polynomials into two classes, according to whether they are polynomials of separable or non-separable digraphs. Some of these polynomials are polynomials of both separable and non-separable digraphs. The number of such polynomials is given in column 6 (COMM) in Table~\ref{table:factorability}. One such example for digraphs of order three is shown in Figure~\ref{fig:dig3_COMM}, where the two digraphs have the same greedoid polynomial $ (1+t)(1+z) $.
\begin{figure}[H]
	\centering
	\resizebox{0.4\textwidth}{!}
	{%
		\begin{tikzpicture}
		[every path/.style={color=black, line width=1.2pt}, 
		every node/.style={draw, circle, line width=1.2pt, inner sep=3pt},
		bend angle=45]
		
		\node	(0) at (0,1.5)	{$ $};
		\node	(1) at (2,1.5)	{$ $};
		\node	(2) at (1,0)	{$ r $};
		\path	(0) edge [->] (1);
		\path	(2) edge [->] (1);
		\node[draw=none] at (1,-0.8)	{$ (a) $};
		
		\node	(3) at (4,1.5)	{$ $};
		\node	(4) at (6,1.5)	{$ $};
		\node	(5) at (5,0)	{$ r $};
		\path	(3) edge [->] (5);
		\path	(5) edge [->] (4);
		\node[draw=none] at (5,-0.8)	{$ (b) $};
		\end{tikzpicture}
	}%
	\caption{Digraphs that have the same greedoid polynomial where (a) is non-separable and (b) is separable}
	\label{fig:dig3_COMM}
\end{figure}
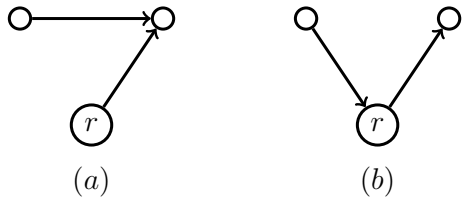

We are interested in non-separable digraphs that can be GM-factorised, especially those digraphs that have greedoid polynomials that are not the same as polynomials of any separable digraph. The numbers of greedoid polynomials of these digraphs are given in column PFNSU in Table~\ref{table:factorability}, and examples of such rooted digraphs of order two and three are given in Figure~\ref{fig:dig2_ndsf} and Figure~\ref{fig:dig3_ndsf}, respectively. It is easy to verify that the greedoid polynomial of the rooted digraph in Figure~\ref{fig:dig2_ndsf} is $ (1 + t)(1 + z) $. The greedoid polynomials of rooted digraphs in Figure~\ref{fig:dig3_ndsf} are (from left to right starting from the first row) given in Table~\ref{table:greedoid_polynomial_order_3}.
\begin{figure}[H]
	\centering
	\resizebox{0.15\textwidth}{!}
	{%
		\begin{tikzpicture}
		[every path/.style={color=black, line width=1.2pt}, 
		every node/.style={draw, circle, line width=1.2pt, inner sep=3pt},
		bend angle=45]
		
		\node	(0) at (0,0)	{$ r $};
		\node	(1) at (2,0)	{$ $}
			edge[->, bend left]		(0)
			edge[<-, bend right]	(0);
		\end{tikzpicture}
	}%
	\caption{The non-separable digraph of order two that GM-factorises}
	\label{fig:dig2_ndsf}
\end{figure}
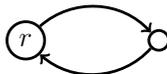

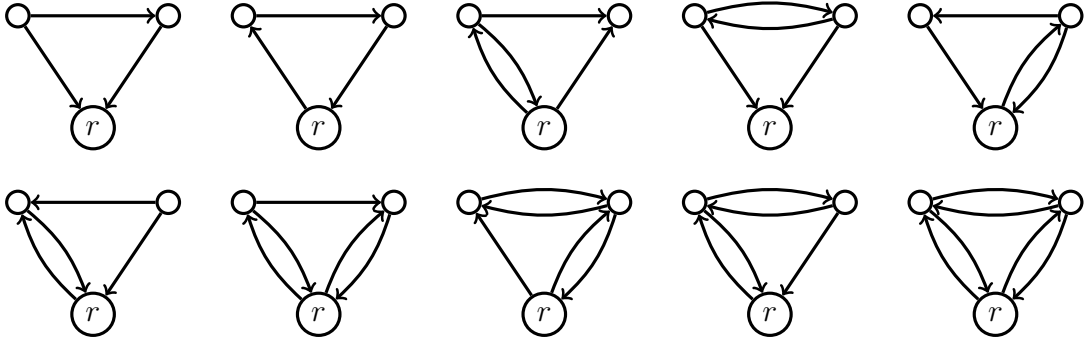
\begin{figure}[H]
	\centering
	\resizebox{0.94\textwidth}{!}
	{%
		\begin{tikzpicture}
		[every path/.style={color=black, line width=1.2pt}, 
		every node/.style={draw, circle, line width=1.2pt, inner sep=3pt},
		bend angle=15]
		
		\node	(0) at (0,1.5)	{$ $};
		\node	(1) at (2,1.5)	{$ $};
		\node	(2) at (1,0)	{$ r $};
		\path	(0) edge [->] (1);
		\path	(0) edge [->] (2);
		\path	(1) edge [->] (2);
		
		\node	(3) at (3,1.5)	{$ $};
		\node	(4) at (5,1.5)	{$ $};
		\node	(5) at (4,0)	{$ r $};
		\path	(3) edge [->] (4);
		\path	(4) edge [->] (5);
		\path	(5) edge [->] (3);
		
		\node	(6) at (6,1.5)	{$ $};
		\node	(7) at (8,1.5)	{$ $};
		\node	(8) at (7,0)	{$ r $};
		\path	(6) edge [->] (7);
		\path	(8) edge [->] (7);
		\path	(6) edge [->, bend left] (8);
		\path	(8) edge [->, bend left] (6);
		
		\node	(9) at (9,1.5)	{$ $};
		\node	(10) at (11,1.5)	{$ $};
		\node	(11) at (10,0)	{$ r $};
		\path	(9) edge [->] (11);
		\path	(10) edge [->] (11);
		\path	(9) edge [->, bend left] (10);
		\path	(10) edge [->, bend left] (9);
		
		\node	(12) at (12,1.5)	{$ $};
		\node	(13) at (14,1.5)	{$ $};
		\node	(14) at (13,0)	{$ r $};
		\path	(13) edge [->] (12);
		\path	(12) edge [->] (14);
		\path	(13) edge [->, bend left] (14);
		\path	(14) edge [->, bend left] (13);
		
		\node	(15) at (0,-1)	{$ $};
		\node	(16) at (2,-1)	{$ $};
		\node	(17) at (1,-2.5)	{$ r $};
		\path	(16) edge [->] (15);
		\path	(16) edge [->] (17);
		\path	(15) edge [->, bend left] (17);
		\path	(17) edge [->, bend left] (15);
		
		\node	(18) at (3,-1)	{$ $};
		\node	(19) at (5,-1)	{$ $};
		\node	(20) at (4,-2.5)	{$ r $};
		\path	(18) edge [->] (19);
		\path	(18) edge [->, bend left] (20);
		\path	(20) edge [->, bend left] (18);
		\path	(19) edge [->, bend left] (20);
		\path	(20) edge [->, bend left] (19);
		
		\node	(21) at (6,-1)	{$ $};
		\node	(22) at (8,-1)	{$ $};
		\node	(23) at (7,-2.5)	{$ r $};
		\path	(23) edge [->] (21);
		\path	(21) edge [->, bend left] (22);
		\path	(22) edge [->, bend left] (21);
		\path	(22) edge [->, bend left] (23);
		\path	(23) edge [->, bend left] (22);
		
		\node	(24) at (9,-1)	{$ $};
		\node	(25) at (11,-1)	{$ $};
		\node	(26) at (10,-2.5)	{$ r $};
		\path	(25) edge [->] (26);
		\path	(24) edge [->, bend left] (25);
		\path	(25) edge [->, bend left] (24);
		\path	(24) edge [->, bend left] (26);
		\path	(26) edge [->, bend left] (24);
		
		\node	(27) at (12,-1)	{$ $};
		\node	(28) at (14,-1)	{$ $};
		\node	(29) at (13,-2.5)	{$ r $};
		\path	(27) edge [->, bend left] (28);
		\path	(28) edge [->, bend left] (27);
		\path	(27) edge [->, bend left] (29);
		\path	(29) edge [->, bend left] (27);
		\path	(28) edge [->, bend left] (29);
		\path	(29) edge [->, bend left] (28);
		\end{tikzpicture}
	}%
	\caption{Ten of the 16 non-separable digraphs (one for each of the ten different greedoid polynomials) of order three that GM-factorise}
	\label{fig:dig3_ndsf}
\end{figure}

\begin{table}[H]
	\centering
	\begin{tabular}{|rl|c|}
		\hline
		\multicolumn{2}{|c|}{\multirow{2}{*}{Greedoid polynomials}} & \multicolumn{1}{c|}{Number of non-separable}\\
		&& \multicolumn{1}{c|}{rooted digraphs of order three}\\
		\hline \hline
		
		1. & $ (1 + z)^{3} $ & 2\\
		
		2. & $ (1 + z)(1 + t + t^{2} + t^{2}z) $ & 3\\
		
		3. & $ (1 + z)(2 + 2t + t^{2} + z + tz + t^{2}z) $ & 2\\
		
		4. & $ (1 + z)^{4} $ & 1\\
		
		5. & $ (1 + z)^{2}(1 + t + t^{2} + t^{2}z) $ & 3\\
		
		6. & $ (1 + t)(1 + z)^{3} $ & 1\\
		
		7. & $ (1 + z)^{2}(2 + 2t + t^{2} + z + tz + t^{2}z) $ & 1\\
		
		8. & $ (1 + z)^{2}(3 + 2t + t^{2} + z + t^{2}z) $ & 1\\
		
		9. & $ (1 + z)^{3}(1 + t + t^{2} + t^{2}z) $ & 1\\
		
		10. & $ (1 + z)^{3}(3 + 2t + t^{2} + z + t^{2}z) $ & 1\\	
		\hline
	\end{tabular}
	\caption{Greedoid polynomials of non-separable digraphs of order three that GM-factorise and these polynomials are not the same as polynomials of any separable digraph of order three, and the numbers of associated non-separable digraphs (making 16 non-separable rooted digraphs altogether)}
	\label{table:greedoid_polynomial_order_3}
\end{table}

\subsection{$ 2 $-nonbasic and $ 1 $-primary Digraphs}

We investigate greedoid polynomials that contain nonbasic and primary GM-factors. Details are given in Table~\ref{table:nonbasic_primary} (see Table~\ref{table:abbreviation_table_3} for the list of abbreviations and Figure~\ref{fig:VennDiagram_nonbasic_primary} for the corresponding Venn diagram). For rooted digraphs up to order six, each $ 1 $-primary digraph is a $ 2 $-nonbasic digraph, and each totally $ 1 $-primary digraph is a totally $ 2 $-nonbasic digraph.
\begin{table}[H]
	\centering
	\begin{tabular}{|l|l|}
		\hline
		\multicolumn{1}{|c}{Abbreviation} & \multicolumn{1}{|c|}{Description}\\
		\hline \hline
		
		2-NB & Number of $ 2 $-nonbasic digraphs\\
		\hline
		2-TNB & Number of totally $ 2 $-nonbasic digraphs\\
		\hline
		1-P & Number of $ 1 $-primary digraphs\\
		\hline
		1-TP & Number of totally $ 1 $-primary digraphs\\
				
		\hline
	\end{tabular}
	\caption{Abbreviations for Figure~\ref{fig:VennDiagram_nonbasic_primary} and Table~\ref{table:nonbasic_primary}}
	\label{table:abbreviation_table_3}
\end{table}

\begin{figure}[H]
	\centering
	\includegraphics[scale=1]{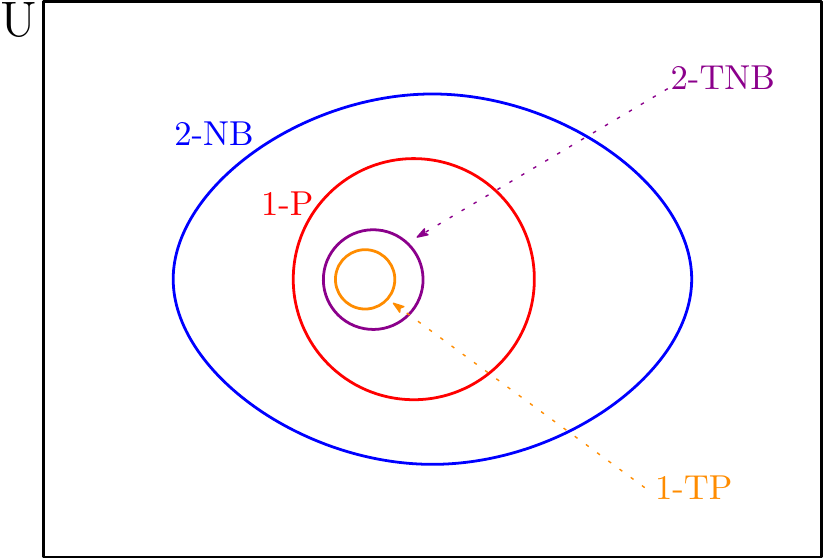}
	\caption{Venn diagram that represents four types of digraphs in Table~\ref{table:nonbasic_primary} where U is the set of digraphs (up to order six) that can be GM-factorised}
	\label{fig:VennDiagram_nonbasic_primary}
\end{figure}

\begin{table}[H]
	\centering
	\begin{tabular}{|c|r|r|r|r|}
		\hline
		\multicolumn{1}{|c}{$ n $} & \multicolumn{1}{|c}{$ 2 $-NB} & \multicolumn{1}{|c}{$ 2 $-TNB} & \multicolumn{1}{|c}{$ 1 $-P} & \multicolumn{1}{|c|}{$ 1 $-TP}\\
		\hline \hline
		
		1 & 0 & 0 & 0 & 0\\
		
		2 & 0 & 0 & 0 & 0\\
		
		3 & 0 & 0 & 0 & 0\\
		
		4 & 0 & 0 & 0 & 0\\
		
		5 & 120 & 0 & 0 & 0\\
		
		6 & 12348 & 15 & 1252 & 9\\	
		\hline
	\end{tabular}
	\caption{Numbers of the four types of non-separable digraphs (up to order six) that can be GM-factorised}
	\label{table:nonbasic_primary}
\end{table}

All rooted digraphs up to order four either have one nonbasic GM-factor or only basic GM-factors in their polynomials. There are 120 rooted digraphs of order five that have greedoid polynomials with at least two nonbasic GM-factors. The number of distinct greedoid polynomials of these 120 rooted digraphs is 34. Further examination showed that the number of nonbasic GM-factors in these polynomials is exactly two. Nonetheless, 117 of the 120 rooted digraphs have greedoid polynomials that contain at least one basic GM-factor, and the remaining three are separable digraphs (as shown in Figure~\ref{fig:order5_nonbasic_but_DS}).
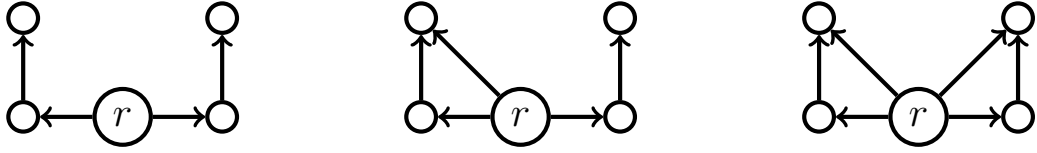
\begin{figure}[H]
	\centering
	\resizebox{0.9\textwidth}{!}
	{%
		\begin{tikzpicture}
		[every path/.style={color=black, line width=1.2pt}, 
		every node/.style={draw, circle, line width=1.2pt, inner sep=3pt},
		bend angle=45]
		
		\node	(0) at (0,0)	{$ $};
		\node	(1) at (0,1)	{$ $};
		\node	(2) at (1,0)	{$ r $};
		\node	(3) at (2,0)	{$ $};
		\node	(4) at (2,1)	{$ $};
		\path	(2) edge [->] (0);
		\path	(0) edge [->] (1);
		\path	(2) edge [->] (3);
		\path	(3) edge [->] (4);
		
		\node	(5) at (4,0)	{$ $};
		\node	(6) at (4,1)	{$ $};
		\node	(7) at (5,0)	{$ r $};
		\node	(8) at (6,0)	{$ $};
		\node	(9) at (6,1)	{$ $};
		\path	(7) edge [->] (5);
		\path	(5) edge [->] (6);
		\path	(7) edge [->] (6);
		\path	(7) edge [->] (8);
		\path	(8) edge [->] (9);
		
		\node	(10) at (8,0)	{$ $};
		\node	(11) at (8,1)	{$ $};
		\node	(12) at (9,0)	{$ r $};
		\node	(13) at (10,0)	{$ $};
		\node	(14) at (10,1)	{$ $};
		\path	(12) edge [->] (10);
		\path	(10) edge [->] (11);
		\path	(12) edge [->] (11);
		\path	(12) edge [->] (13);
		\path	(13) edge [->] (14);
		\path	(12) edge [->] (14);		
		\end{tikzpicture}
	}%
	\caption{Three separable digraphs of order five that have two nonbasic GM-factors}
	\label{fig:order5_nonbasic_but_DS}
\end{figure}
\noindent
Hence, there exist no totally $ 2 $-nonbasic digraphs of order five. In addition, none of the polynomials of these 120 rooted digraphs contains a primary GM-factor. This implies that none of the rooted digraphs up to order five are $ k $-primary, for $ k \ge 1 $. Each of the GM-factors of greedoid polynomials of rooted digraph up to order five is either basic, or is a GM-factor of some greedoid polynomials of rooted digraphs of smaller order. 

There are 12348 rooted digraphs of order six that have greedoid polynomials with at least two nonbasic GM-factors. The number of distinct greedoid polynomials of these 12348 rooted digraphs is 837. A quick search showed that all these digraphs are $ 2 $-nonbasic. We found that 15 of these rooted digraphs are totally $ 2 $-nonbasic. One of the totally $ 2 $-nonbasic digraphs $ D_{1} $ of order six is shown in Figure~\ref{fig:dig6_totally_nonbasic} and its greedoid polynomial is as follows:
\begin{equation*}
	f(D_{1}) = (1 + t + t^{2} + t^{2}z)(2 + 2t + t^{2} + t^{3} + z + tz + t^{2}z + 
	3t^{3}z + 3t^{3}z^{2} + t^{3}z^{3}).
\end{equation*}

\begin{figure}[H]
	\centering
	\resizebox{0.4\textwidth}{!}
	{%
		\begin{tikzpicture}
		[every path/.style={color=black, line width=1.2pt}, 
		every node/.style={draw, circle, line width=1.2pt, inner sep=3pt},
		bend angle=45]
		
		\node	(0) at (1.5,2)	{$ r $};
		\node	(1) at (0,1)	{$ $};
		\node	(2) at (1.5,0)	{$ $};
		\node	(3) at (3,1)	{$ $};
		\node	(4) at (4.5,1)	{$ $};
		\node	(5) at (4.5,0)	{$ $};
		\path	(0) edge [->] (1);
		\path	(0) edge [->] (3);
		\path	(1) edge [->] (2);
		\path	(3) edge [->] (2);
		\path	(3) edge [->] (4);
		\path	(4) edge [->] (5);
		\end{tikzpicture}
	}%
	\caption{A totally $ 2 $-nonbasic digraph of order six}
	\label{fig:dig6_totally_nonbasic}
\end{figure}
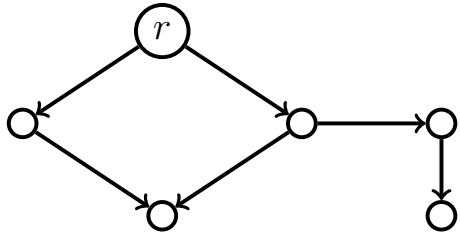
\noindent
Both of the nonbasic GM-factors of $ f(D_{1}) $ are greedoid polynomials of rooted digraphs $ G $ and $ H $ that have order three and four, respectively (see Figure~\ref{fig:digraphs_smaller_order}). We have $ f(G) = 1 + t + t^{2} + t^{2}z $ and $ f(H) = 2 + 2t + t^{2} + t^{3} + z + tz + t^{2}z + 3t^{3}z + 3t^{3}z^{2} + t^{3}z^{3} $. However, $ D_{1} $ is a non-separable digraph and hence not the direct sum of $ G $ and $ H $.
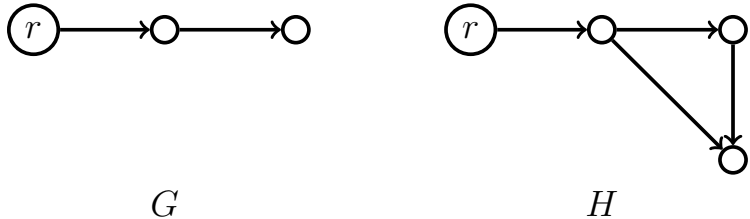
\begin{figure}[H]
	\centering
	\resizebox{0.65\textwidth}{!}
	{%
		\begin{tikzpicture}
		[every path/.style={color=black, line width=1.2pt}, 
		every node/.style={draw, circle, line width=1.2pt, inner sep=3pt},
		bend angle=45]
		
		\node	(0) at (0,2)	{$ r $};
		\node	(1) at (1.5,2)	{$ $};
		\node	(2) at (3,2)	{$ $};
		\path	(0) edge [->] (1);
		\path	(1) edge [->] (2);
		\node[draw=none]	() at (1.5,0)	{$ G $};
		
		\node	(3) at (5,2)	{$ r $};
		\node	(4) at (6.5,2)	{$ $};
		\node	(5) at (8,2)	{$ $};
		\node	(6) at (8,0.5)	{$ $};
		\path	(3) edge [->] (4);
		\path	(4) edge [->] (5);
		\path	(4) edge [->] (6);
		\path	(5) edge [->] (6);
		\node[draw=none]	() at (6.5,0)	{$ H $};
		
		\end{tikzpicture}
	}%
	\caption{Rooted digraphs $ G $ and $ H $}
	\label{fig:digraphs_smaller_order}
\end{figure}

There are also 1252 rooted digraphs of order six that have greedoid polynomials with one primary GM-factor, and all these digraphs are non-separable. However, only nine of them are totally $ 1 $-primary digraphs. One of the totally $ 1 $-primary digraphs $ D_{2} $ of order six is shown in Figure~\ref{fig:dig6_totally_primary} and it has the following greedoid polynomial:
\begin{equation*}\label{eqn:totally_primaryPoly}
	f(D_{2}) = (1 + t + t^{2} + t^{2}z)(4 + 3t + t^{2} + t^{3} + 4z + 2tz + t^{2}z + 4t^{3}z + z^{2} + 6t^{3}z^{2} + 4t^{3}z^{3} + t^{3}z^{4}).
\end{equation*}

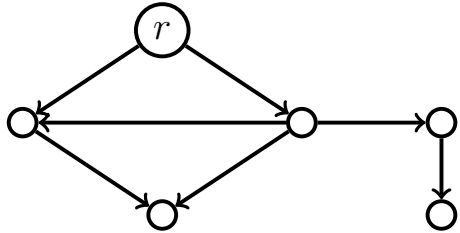
\begin{figure}[H]
	\centering
	\resizebox{0.4\textwidth}{!}
	{%
		\begin{tikzpicture}
		[every path/.style={color=black, line width=1.2pt}, 
		every node/.style={draw, circle, line width=1.2pt, inner sep=3pt},
		bend angle=45]
		
		\node	(0) at (1.5,2)	{$ r $};
		\node	(1) at (0,1)	{$ $};
		\node	(2) at (1.5,0)	{$ $};
		\node	(3) at (3,1)	{$ $};
		\node	(4) at (4.5,1)	{$ $};
		\node	(5) at (4.5,0)	{$ $};
		\path	(0) edge [->] (1);
		\path	(0) edge [->] (3);
		\path	(3) edge [->] (1);
		\path	(1) edge [->] (2);
		\path	(3) edge [->] (2);
		\path	(3) edge [->] (4);
		\path	(4) edge [->] (5);
		\end{tikzpicture}
	}%
	\caption{A totally $ 1 $-primary digraph of order six}
	\label{fig:dig6_totally_primary}
\end{figure}
\noindent
The totally $ 1 $-primary digraph $ D_{2} $ GM-factorises into one nonbasic GM-factor $ 1 + t + t^{2} + t^{2}z $ and one primary GM-factor $ 4 + 3t + t^{2} + t^{3} + 4z + 2tz + t^{2}z + 4t^{3}z + z^{2} + 6t^{3}z^{2} + 4t^{3}z^{3} + t^{3}z^{4} $. The GM-factor $ 1 + t + t^{2} + t^{2}z $ is not primary as it is the greedoid polynomial of the rooted digraph $ G $ in Figure~\ref{fig:digraphs_smaller_order}. Note that $ D_{2} $ is also a totally $ 2 $-nonbasic digraph since every primary GM-factor is a nonbasic GM-factor.

The fact that a greedoid polynomial of a rooted digraph is not divisible by $ 1 + z $ implies that the associated rooted digraph has neither a directed cycle nor a greedoid loop. Our results show that there exist some non-separable digraphs (of order six) that GM-factorise into only nonbasic GM-factors, or both nonbasic and primary GM-factors. This implies that the multiplicative direct sum property, and the existence of greedoid loops and directed cycles, are not the only characteristics that determine if greedoid polynomials of rooted digraphs factorise.

\subsection{An Infinite Family}

Lastly, we show that there exists an infinite family of digraphs where their greedoid polynomials factorise into at least two nonbasic GM-factors. We first characterise greedoid polynomials of two classes of rooted digraphs.

Let $ P_{m,v_{0}} $ be a \emph{directed path} $ v_{0}v_{1}\ldots v_{m} $ of size $ m \ge 0 $ \emph{rooted} at $ v_{0} $, and $ C_{m,v_{0}} $ be a \emph{directed cycle} $ v_{0}v_{1}\ldots v_{m-1}v_{0} $ of size $ m \ge 1 $ \emph{rooted} at $ v_{0} $. For convenience, we usually write $ P_{m} $ for $ P_{m,v_{0}} $ and $ C_{m} $ for $ C_{m,v_{0}} $.
\begin{lemma}\label{lem:directed_path}
	\[ f(P_{m};t,z) = 1 + \frac{t(1-(t(1+z))^{m})}{1-t(1+z)}. \]
\end{lemma}

\begin{proof}
	By induction on the number of edges.
\end{proof}

Let $ G $ be a rooted undirected graph and $ X \subseteq E(G) $. The \emph{rank} $ r(X) $ of $ X $ is defined as $ r(X)=\hbox{max}\{\abs{A}:A \subseteq X, A \text{ is a rooted subtree}\} $. Let $ F $ be the set of subtrees of $ G $ containing the root vertex. Korte and Lov\'{a}sz~\cite{KorteLovasz1984} showed that $ (G,F) $ is a greedoid called the \emph{branching greedoid} of $ G $.

Suppose $ Q_{m} $ is an \emph{undirected path} $ v_{0}v_{1}\ldots v_{m} $ of size $ m \ge 0 $ \emph{rooted} at either $ v_{0} $ or $ v_{m} $. Then $ f(P_{m};t,z) = f(Q_{m};t,z) $, since there is a rank-preserving bijection between $ 2^{E(P_{m})} $ and $ 2^{E(Q_{m})} $.

\begin{lemma}\label{lem:directed_cycle}
	\[ f(C_{m};t,z) = (1 + z) f(P_{m-1};t,z). \]
\end{lemma}

\begin{proof}
	By induction on the number of edges.
\end{proof}

Gordon gave a formula for the greedoid polynomials of rooted undirected cycles in~\cite{Gordon2008}. Those polynomials are different to the ones given by Lemma~\ref{lem:directed_cycle}.

We now give an infinite family of digraphs where their greedoid polynomials factorise into at least two nonbasic GM-factors, extending the example in Figure~\ref{fig:dig6_totally_nonbasic}.

\begin{lemma}\label{lem:infinite_family}
	There exists an infinite family of non-separable digraphs $ D $ that have at least two nonbasic GM-factors, where
	\begin{equation*}
		f(D) = f(P_{k+1}) \left( f(C_{k+1}) + f(P_{k+1}) + t^{k+2}(1+z)^{k+2} \right) ,\ \text{for}\ k \ge 1.
	\end{equation*}
\end{lemma}

\begin{proof}
	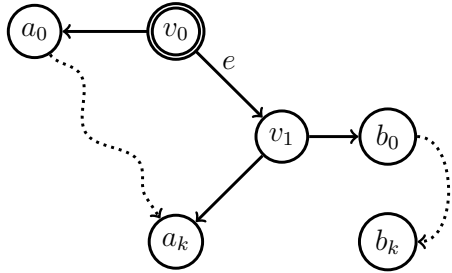
\begin{figure}[H]
		\centering
		\resizebox{0.4\textwidth}{!}
		{%
			\begin{tikzpicture}
			[every path/.style={color=black, line width=1.2pt}, 
			every node/.style={draw, circle, line width=1.2pt, inner sep=3pt},
			bend angle=45]
			
			\node[double]	(0) at (1.5,3)	{$ v_{0} $};
			\node	(1) at (-0.5,3)	{$ a_{0} $};
			\node	(2) at (1.5,0)	{$ a_{k} $};
			\node	(3) at (3,1.5)	{$ v_{1} $};
			\node	(4) at (4.5,1.5)	{$ b_{0} $};
			\node	(5) at (4.5,0)	{$ b_{k} $};
			\path	(0) edge [->] (1);
			\path	(0) edge [->] node [above, draw=none] {$ e $} (3);
			\path[dotted]	(1) edge [->,decorate,decoration={snake,amplitude=2mm,segment length=15mm}] (2);
			\path	(3) edge [->] (2);
			\path	(3) edge [->] (4);
			\path[dotted]	(4) edge [->,out=0,in=0] (5);
			\end{tikzpicture}
		}%
		\caption{The digraph $ D $ in the proof of Lemma~\ref{lem:infinite_family}}
		\label{fig:digraph_infinite_family}
	\end{figure}

	Let $ D $ be the non-separable digraph rooted at vertex $ v_{0} $ shown in Figure~\ref{fig:digraph_infinite_family}, where $ a_{0} \ldots a_{k} $ and $ b_{0} \ldots b_{k} $ are two directed paths in $ D $ of length $ k \ge 1 $ starting at $ a_{0} $ and $ b_{0} $, respectively. To compute the greedoid polynomial of $ D $ by using Proposition~\ref{pro:recurrence}, we first choose the edge $ e=v_{0}v_{1} $. By deleting and contracting $ e $, we obtain the digraphs $ D_{1}=D/e $ and $ D_{2}=D\setminus e $ as shown in Figure~\ref{fig:deletion_contraction}.
	\begin{figure}[t]
		\centering
		\resizebox{0.8\textwidth}{!}
		{%
			\begin{tikzpicture}
			[every path/.style={color=black, line width=1.2pt}, 
			every node/.style={draw, circle, line width=1.2pt, inner sep=3pt},
			bend angle=45]
			
			\node	(1) at (-0.5,3)	{$ a_{0} $};
			\node	(2) at (1.5,0)	{$ a_{k} $};
			\node[double]	(3) at (2.2,2.5)	{$  v_{0} $};
			\node	(4) at (3.7,2.5)	{$  b_{0} $};
			\node	(5) at (3.7,1)	{$  b_{k} $};
			\path	(3) edge [->] (1);
			\path[dotted]	(1) edge [->,decorate,decoration={snake,amplitude=2mm,segment length=15mm}] (2);
			\path	(3) edge [->] (2);
			\path	(3) edge [->] (4);
			\path[dotted]	(4) edge [->,out=0,in=0] (5);
			\node[draw=none]	() at (1.8,-1.2)	{$ (a)\ D_{1}=D/e $};
			
			\node[double]	(6) at (8.5,3)	{$  v_{0} $};
			\node	(7) at (6.5,3)	{$  a_{0} $};
			\node	(8) at (8.5,0)	{$  a_{k} $};
			\node	(9) at (10,1.5)	{$  v_{1} $};
			\node	(10) at (11.5,1.5)	{$  b_{0} $};
			\node	(11) at (11.5,0)	{$  b_{k} $};
			\path	(6) edge [->] (7);
			\path[dotted]	(7) edge [->,decorate,decoration={snake,amplitude=2mm,segment length=15mm}] (8);
			\path	(9) edge [->] (8);
			\path	(9) edge [->] (10);
			\path[dotted]	(10) edge [->,out=0,in=0] (11);
			\node[draw=none]	() at (9,-1.2)	{$ (b)\ D_{2}=D\setminus e $};
			\end{tikzpicture}
		}%
		\vspace{-1cm}
		\caption{Two minors $ D/e $ and $ D\setminus e $ of $ D $}
		\label{fig:deletion_contraction}
	\end{figure}
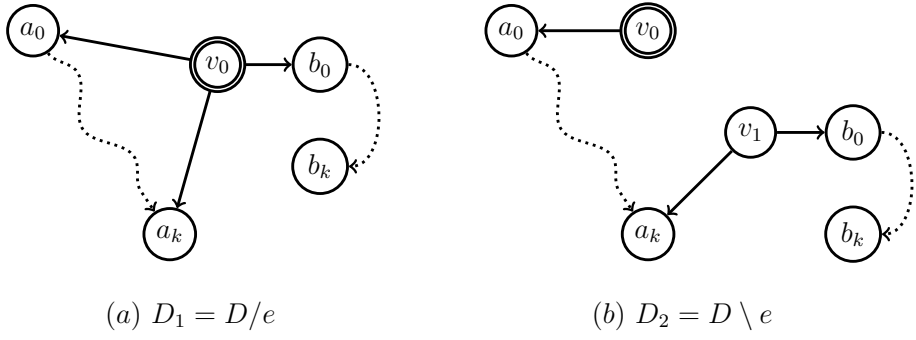
	
	Note that $ D_{1} $ is a separable digraph rooted at $ v_{0} $. Let $ R=\{v_{0},a_{0},\ldots,a_{k}\} \subset V(D_{1}) $, $ S=\{v_{0},b_{0},\ldots,b_{k}\} \subset V(D_{1}) $ and $ T=\{v_{0},a_{0},\ldots,a_{k}\} \subset V(D_{2}) $. Suppose $ A=D_{1}[R] $ and $ B=D_{1}[S] $ are the subdigraphs of $ D_{1} $ induced by $ R $ and $ S $ respectively, and $ C=D_{2}[T] $ is the subdigraph of $ D_{2} $ induced by $ T $. Clearly, $ B \cong C \cong P_{k+1} $. Hence we have $ f(B) = f(C) = f(P_{k+1}) $. Note that every edge $ g \in E(D_{2})\setminus E(C) $ is a greedoid loop, and $ \abs{E(D_{2})\setminus E(C)} = k+2 $. By using the recurrence formula, we have
	\begin{align*}
		f(D) & = f(D/e) + t^{r(D)-r(D\setminus e)}f(D\setminus e)\\
			& = f(A) \cdot f(B) + t^{(2k+3)-(k+1)} f(C) \cdot (1+z)^{k+2}\\
			& = f(P_{k+1}) \left( f(A) + t^{k+2}(1+z)^{k+2} \right) \tag[.]{since $ f(B)=f(C)=f(P_{k+1}) $}
	\end{align*}
	
	It remains to show that $ f(A) $ can be expressed in terms of $ f(P_{k}) $ and $ f(C_{k}) $. By taking $ h=v_{0}a_{k} \in E(A) $ (see Figure~\ref{fig:induced_subdigraph}) as the outgoing edge in the recurrence formula, we have
	\begin{figure}[H]
		\centering
		\resizebox{0.2\textwidth}{!}
		{%
			\begin{tikzpicture}
			[every path/.style={color=black, line width=1.2pt}, 
			every node/.style={draw, circle, line width=1.2pt, inner sep=3pt},
			bend angle=45]
			
			\node	(1) at (-0.5,3)	{$ a_{0} $};
			\node	(2) at (1.5,0)	{$ a_{k} $};
			\node[double]	(3) at (2.2,2.5)	{$  v_{0} $};
			\path	(3) edge [->] (1);
			\path[dotted]	(1) edge [->,decorate,decoration={snake,amplitude=2mm,segment length=15mm}] (2);
			\path	(3) edge [->] node [right, draw=none] {$ h $} (2);
			\end{tikzpicture}
		}%
		\caption{The subdigraph $ A $ of $ D_{1} $ induced by $ R $}
		\label{fig:induced_subdigraph}
	\end{figure}
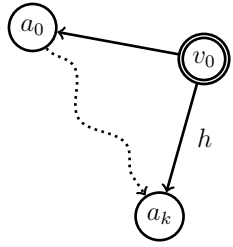
	
	\begin{align*}
		f(A) & = f(A/h) + t^{r(A)-r(A\setminus h)} f(A\setminus h)\\
			& = f(C_{k+1}) + t^{(k+1)-(k+1)} f(P_{k+1}) \tag{since $ A/h \cong C_{k+1} $ and $ A\setminus h \cong P_{k+1} $}\\
			& = f(C_{k+1}) + f(P_{k+1}).
	\end{align*}
	
	Therefore,
	\begin{align*}
		f(D) & = f(P_{k+1}) \left( f(C_{k+1}) + f(P_{k+1}) + t^{k+2}(1+z)^{k+2} \right).
	\end{align*}
	Clearly, both factors of $ f(D) $ are nonbasic GM-factors. Since $ D $ is non-separable and $ k \ge 1 $, the proof is complete.
\end{proof}

We extend the infinite family in Lemma~\ref{lem:infinite_family}, and characterise the greedoid polynomials of a new infinite family, as follows.
\begin{theorem}\label{thm:infinite_family}
	There exists an infinite family of non-separable digraphs $ D $ that have at least two nonbasic GM-factors, where
	\begin{equation*}
		f(D) = f(P_{k+1}) \left( f(C_{k+1}) + f(P_{k+1}) + \frac{t^{k+2}(1+z)^{k+2}(1-(t(1+z))^{\ell})}{1-t(1+z)} \right),\ \text{for}\ k,\ell \ge 1.
	\end{equation*}
\end{theorem}

\begin{proof}
	\begin{figure}[H]
		\centering
		\resizebox{0.4\textwidth}{!}
		{%
			\begin{tikzpicture}
			[every path/.style={color=black, line width=1.2pt}, 
			every node/.style={draw, circle, line width=1.2pt, inner sep=3pt},
			bend angle=45]
			
			\node[double]	(0) at (1.5,3)	{$ v_{0} $};
			\node	(1) at (3,3)	{$ v_{1} $};
			\node	(2) at (7,3)	{$ v_{\ell} $};
			
			\node	(3) at (1.5,1.5)	{$ a_{0} $};
			\node	(4) at (3.5,0)	{$ a_{k} $};
			
			\node	(5) at (7,1.5)	{$ b_{0} $};
			\node	(6) at (5,0)	{$ b_{k} $};
			
			\path	(0) edge [->] (3);
			\path	(0) edge [->] node [above, draw=none] {$ e $} (1);
			\path[dotted]	(3) edge [->,decorate,decoration={snake,amplitude=2mm,segment length=15mm}] (4);
			\path[dotted]	(1) edge [->,decorate,decoration={snake,amplitude=2mm,segment length=15mm}] (2);
			\path	(2) edge [->] (4);
			\path	(2) edge [->] (5);
			\path[dotted]	(5) edge [->,decorate,decoration={snake,amplitude=2mm,segment length=15mm}] (6);
			\end{tikzpicture}
		}%
		\caption{The digraph $ D $ in the proof of Theorem~\ref{thm:infinite_family}}
		\label{fig:digraph_infinite_family_ext}
	\end{figure}
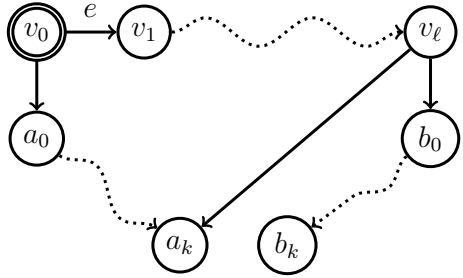
	
	Let $ D $ be the non-separable digraph rooted at vertex $ v_{0} $ shown in Figure~\ref{fig:digraph_infinite_family_ext}, where $ L = v_{0}\ldots v_{\ell} $ is a directed path in $ D $ of length $ \ell \ge 1 $ starting at $ v_{0} $. We proceed by induction on the length $ \ell $ of $ L $.
	
	For the base case, suppose $ \ell=1 $. By Lemma~\ref{lem:infinite_family}, we have
	\begin{align*}
		f(D) & = f(P_{k+1}) \left( f(C_{k+1}) + f(P_{k+1}) + t^{k+2}(1+z)^{k+2} \right)\\
			& = f(P_{k+1}) \left( f(C_{k+1}) + f(P_{k+1}) + \frac{t^{k+2}(1+z)^{k+2}(1-(t(1+z))^{\ell})}{1-t(1+z)} \right),
	\end{align*}
	and the result for $ \ell=1 $ follows.
	
	Assume that $ \ell > 1 $ and the result holds for every $ r < \ell $.
	
	Let $ e = v_{0}v_{1} \in E(D) $. By applying the deletion-contraction recurrence in Proposition~\ref{pro:recurrence} on $ e $, we obtain the digraphs $ D_{1}=D/e $ and $ D_{2}=D\setminus e $ as shown in Figure~\ref{fig:deletion_contraction_ext}.
	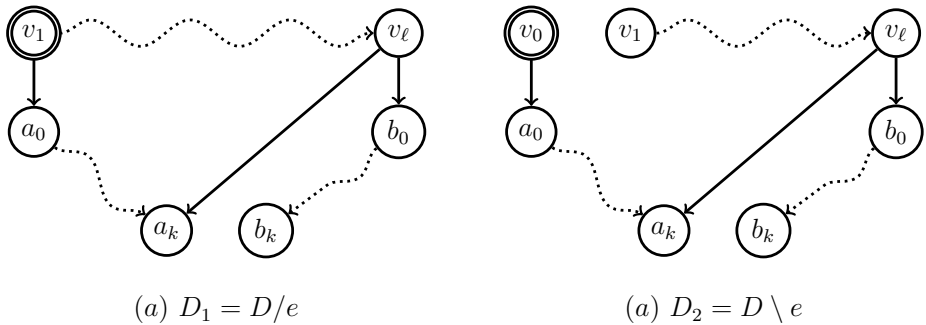
\begin{figure}[H]
		\centering
		\resizebox{0.8\textwidth}{!}
		{%
			\begin{tikzpicture}
			[every path/.style={color=black, line width=1.2pt}, 
			every node/.style={draw, circle, line width=1.2pt, inner sep=3pt},
			bend angle=45]
			
			\node[double]	(0) at (1.5,3)	{$ v_{1} $};
			\node	(2) at (7,3)	{$ v_{\ell} $};
			
			\node	(3) at (1.5,1.5)	{$ a_{0} $};
			\node	(4) at (3.5,0)	{$ a_{k} $};
			
			\node	(5) at (7,1.5)	{$ b_{0} $};
			\node	(6) at (5,0)	{$ b_{k} $};
						
			\path	(0) edge [->] (3);
			\path[dotted]	(0) edge [->,decorate,decoration={snake,amplitude=2mm,segment length=15mm}] (2);
			\path[dotted]	(3) edge [->,decorate,decoration={snake,amplitude=2mm,segment length=15mm}] (4);
			\path	(2) edge [->] (4);
			\path	(2) edge [->] (5);
			\path[dotted]	(5) edge [->,decorate,decoration={snake,amplitude=2mm,segment length=15mm}] (6);
			\node[draw=none]	() at (4.25,-1.2)	{$ (a)\ D_{1}=D/e $};
			
			\node[double]	(0) at (9,3)	{$ v_{0} $};
			\node	(1) at (10.5,3)	{$ v_{1} $};
			\node	(2) at (14.5,3)	{$ v_{\ell} $};
			
			\node	(3) at (9,1.5)	{$ a_{0} $};
			\node	(4) at (11,0)	{$ a_{k} $};
			
			\node	(5) at (14.5,1.5)	{$ b_{0} $};
			\node	(6) at (12.5,0)	{$ b_{k} $};
				
			\path	(0) edge [->] (3);
			\path[dotted]	(3) edge [->,decorate,decoration={snake,amplitude=2mm,segment length=15mm}] (4);
			\path[dotted]	(1) edge [->,decorate,decoration={snake,amplitude=2mm,segment length=15mm}] (2);
			\path	(2) edge [->] (4);
			\path	(2) edge [->] (5);
			\path[dotted]	(5) edge [->,decorate,decoration={snake,amplitude=2mm,segment length=15mm}] (6);
			\node[draw=none]	() at (11.75,-1.2)	{$ (a)\ D_{2}=D\setminus e $};
			\end{tikzpicture}
		}%
		\vspace{-0.7cm}
		\caption{Two minors $ D/e $ and $ D\setminus e $ of $ D $}
		\label{fig:deletion_contraction_ext}
	\end{figure}
	
	Note that $ D_{1} $ is a non-separable digraph rooted at $ v_{1} $. Since the directed path $ v_{1} \ldots v_{\ell} $ in $ D_{1} $ has length $ \ell-1 $, we use the inductive hypothesis to obtain $ f(D_{1}) $. Let $ R=\{v_{0},a_{0},\ldots,a_{k}\} \subset V(D_{2}) $, and $ A = D_{2}[R] $ be the subdigraph of $ D_{2} $ induced by $ R $. Clearly, $ A \cong P_{k+1} $. Hence, we have $ f(A) = f(P_{k+1}) $. Note that every edge $ g \in E(D_{2})\setminus E(A) $ is a greedoid loop, and $ \abs{E(D_{2})\setminus E(A)}=k+\ell+1 $. By using the recurrence formula, we have
	\begin{align*}
		f(D) & = f(D/e) + t^{r(D)-r(D\setminus e)}f(D\setminus e)\\
			& = f(P_{k+1}) \left( f(C_{k+1}) + f(P_{k+1}) + \frac{t^{k+2}(1+z)^{k+2}(1-(t(1+z))^{\ell-1})}{1-t(1+z)} \right)\\
			& \quad + t^{(2k+\ell+2)-(k+1)} \left( f(P_{k+1}) \cdot (1+z)^{k+\ell+1} \right)\\
			& = f(P_{k+1}) \bigg( f(C_{k+1}) + f(P_{k+1}) + \left( \frac{t^{k+2}(1+z)^{k+2}(1-(t(1+z))^{\ell-1})}{1-t(1+z)} \right)\\
			& \quad + t^{k+\ell+1} (1+z)^{k+\ell+1} \bigg)\\
			& = f(P_{k+1}) \left( f(C_{k+1}) + f(P_{k+1}) + \frac{t^{k+2}(1+z)^{k+2}(1-(t(1+z))^{\ell})}{1-t(1+z)} \right).
	\end{align*}
\end{proof}


We observe that if every directed path has length at most one in a digraph $ D $ rooted at a vertex $ v $, the greedoid polynomial of $ D $ is trivial. In this scenario, every vertex in $ D $ is either a sink vertex or a source vertex. If $ v $ is a sink vertex, then every edge in $ D $ is a greedoid loop. If $ v $ is a source vertex, every edge that is not incident with $ v $ is a greedoid loop.

In the following theorem, we show that the greedoid polynomial of any digraph that has a directed path of length at least two is a nonbasic GM-factor of the greedoid polynomial of some non-separable digraph. The proof follows similar approaches as in Lemma~\ref{lem:infinite_family} and Theorem~\ref{thm:infinite_family}.
\begin{theorem}\label{thm:infinite_family_any_graph}
	For any digraph $ G $ that has a directed path of length at least two, there exists a non-separable digraph $ D $ where $ f(D) $ has $ f(G) $ as a nonbasic GM-factor.
\end{theorem}

\begin{figure}[H]
	\centering
	\resizebox{0.6\textwidth}{!}
	{%
		\begin{tikzpicture}
		[every path/.style={color=black, line width=1.2pt}, 
		every node/.style={draw, circle, line width=1.2pt, inner sep=3pt},
		bend angle=45]
		
		\node[double]	(0) at (1.5,4)	{$ a_{0} $};
		\node	(1) at (3,4)	{$ v_{1} $};
		\node	(2) at (5.5,4)	{$ a'_{0} $};
		
		\node	(3) at (0,4)	{$ a_{1} $};
		\node	(4) at (0,2.5)	{$ a_{2} $};
		\node	(5) at (2.7,1)	{$ a_{k} $};
		
		\node	(6) at (7,4)	{$ a'_{1} $};
		\node	(7) at (7,2.5)	{$ a'_{2} $};
		\node	(8) at (4.3,1)	{$ a'_{k} $};
		
		\node[draw=none]	(9) at (0.5,2)	{$ $};
		\node[draw=none]	(G) at (1.1,3)	{$ G $};
		\node[draw=none]	(H) at (6.1,3)	{$ G' $};
		
		\path	(0) edge [->] (1);
		\path	(0) edge [->] node [above, draw=none] {$ e $} (1);
		\path[dotted]	(1) edge [->,decorate,decoration={snake,amplitude=2mm,segment length=15mm}] (2);
		\path	(2) edge [->] (5);
		
		\path	(0) edge [->] (3);
		\path	(3) edge [->] (4);			
		\path[dotted]	(4) edge [->,decorate,decoration={snake,amplitude=2mm,segment length=15mm}] (5);			
		\path[dashed]	(0) edge [-,decorate,decoration={snake,amplitude=2mm,segment length=15mm}] (5);
		
		\path	(2) edge [->] (6);
		\path	(6) edge [->] (7);			
		\path[dotted]	(7) edge [->,decorate,decoration={snake,amplitude=2mm,segment length=15mm}] (8);			
		\path[dashed]	(2) edge [-,decorate,decoration={snake,amplitude=2mm,segment length=15mm}] (8);
		\end{tikzpicture}
	}%
	\caption{An illustration of the non-separable digraph $ D $ in Theorem~\ref{thm:infinite_family_any_graph}}
	\label{fig:digraph_infinite_family_any_graph}
\end{figure}
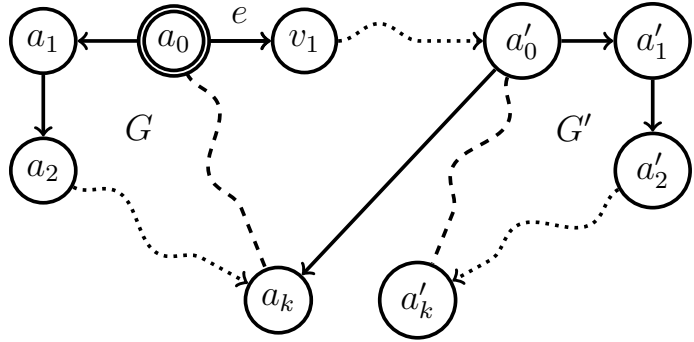

\begin{proof}
 	Let $ G $ be a digraph that has a directed path $ K = a_{0}a_{1} \ldots a_{k} $ of length $ k \ge 2 $, and $ G' $ be a copy of $ G $. The copy of $ K $ in $ G' $ is denoted by $ K' = a'_{0}a'_{1} \ldots a'_{k} $.
 	
 	We construct a non-separable digraph $ D_{\ell} $ using $ G $ and $ G' $, as follows. We first create a directed path $ L = a_{0}v_{1} \ldots v_{\ell-1}a'_{0} $ of length $ \ell $. We add a directed edge $ a'_{0}a_{k} $, and assign $ v_{0} $ as the root vertex of $ D_{\ell} $ (see Figure~\ref{fig:digraph_infinite_family_any_graph}).
 	
 	To show that $ f(G) $ is a nonbasic GM-factor of $ f(D_{\ell}) $, we proceed by induction on the length $ \ell $ of $ L $.
 	
 	For the base case, suppose $ \ell = 1 $. We apply the deletion-contraction recurrence in Proposition~\ref{pro:recurrence} on $ e=a_{0}a'_{0} $. We denote $ a_{0} $ the root vertex of the separable digraph $ D_{1}/e $. We have
 	\begin{align*}
 	 	f(D_{1}) & = f(D_{1}/e) + t^{r(D_{1})-r(D_{1}\setminus e)}f(D_{1}\setminus e)\\
 	 		& = f(G + a_{0}a_{k}) \cdot f(G) + t^{(2r(G)+1)-r(G)} f(G) \cdot (1+z)^{\abs{E(G)}+1}\\
 	 		& = f(G) \left( f(G + a_{0}a_{k}) + t^{r(G)+1} (1+z)^{\abs{E(G)}+1} \right).
 	\end{align*}
 	\noindent
 	Hence, the result for $ \ell=1 $ follows.
 	
 	Assume that $ \ell > 1 $ and the result holds for every $ r < \ell $.
 	
	For the inductive steps, we apply the deletion-contraction recurrence on $ e=v_{0}v_{1} $. We have
	\begin{align*}
		f(D_{\ell}) & = f(D_{\ell}/e) + t^{r(D_{\ell})-r(D_{\ell}\setminus e)}f(D_{\ell}\setminus e)\\
			& = f(D_{\ell}/e) + t^{(2r(G)+\ell)-r(G)} f(G) \cdot (1+z)^{\abs{E(G)}+\ell}\\
			& = f(D_{\ell}/e) + t^{r(G)+\ell} f(G) \cdot (1+z)^{\abs{E(G)}+\ell}.
	\end{align*}
	\noindent
	Note that $ D_{\ell}/e \cong D_{\ell-1} $. By the inductive hypothesis, $ f(D_{\ell}/e) $ has $ f(G) $ as a nonbasic GM-factor. This implies that $ f(D_{\ell}) $ has $ f(G) $ as a nonbasic GM-factor.
\end{proof}

We now have the following corollary.
\begin{corollary}
	Let $ D $ be a non-separable digraph that belongs to the infinite family in Theorem~\ref{thm:infinite_family_any_graph}. By replacing the edge $ a'_{0}a_{k} \in E(D) $ by any digraph $ R $ such that every edge in $ E(R) $ that is incident with $ a_{k} $ is an incoming edge of $ a_{k} $, then $ f(D) $ has $ f(G) $ as a nonbasic GM-factor. \hfill \qed
\end{corollary}

\section{Computational Methods}

All labelled rooted digraphs (without loops and multiple edges, but cycles of size two are allowed) up to order six were provided by Brendan McKay on 28 March 2018 (personal communication from McKay to Farr). Each digraph is given as a list of numbers on one line separated by a single space. The first number is the order of the digraph, the second number is the size of the digraph, and each pair of subsequent numbers represent a directed edge of the digraph. For instance, $ 3\ 2\ 2\ 0\ 2\ 1 $ represents a digraph of order $ 3 $ and size $ 2 $. The directed edges  of the digraph are $ (2,0) $ and $ (2,1) $. Details are as follows: \[ \overbrace{3}^{\text{order}} \underbrace{2}_{\text{size}} \overbrace{2\ 0}^{\text{edge}} \underbrace{2\ 1}_{\text{edge}}. \]

We use the set of numbers $ \{0,1,\ldots,n-1\} $ to represent vertices for each digraph of order $ n $, and an edge list to represent the edge set of each digraph, e.g., $ [[0,1]] $ represents a digraph with a single edge directed from vertex $ 0 $ to vertex $ 1 $.

We use Python 3, Wolfram Mathematica 11 and Bash Shell (Mac OS Version 10.13.4), in computing results for greedoid polynomials of rooted digraphs up to order six.

Algorithms of our programs and the steps used in obtaining our results are given in \cite{YowPHD2019} and \cite{YowFM2018pp_greedoid}.

\section{Concluding Remarks}

In this paper, we presented (i) the results from exhaustive computation of all small rooted digraphs and (ii) the first results of the GM-factorability of greedoid polynomials of rooted digraphs.

We computed the greedoid polynomials for all rooted digraphs up to order six. From Table~\ref{table:unique_poly}, the ratio of PU to T-ISO shows a decreasing trend. We expect that as $ n $ increases, this ratio continues to decrease. Hence, we have the following conjecture.
\begin{conjecture}\label{conj:poly_vs_rooted_digraphs}
	Most rooted digraphs are not determined by their greedoid polynomials.
\end{conjecture}
 
We found that the multiplicative direct sum property, and the existence of greedoid loops and directed cycles, are not the only characteristics that determine if greedoid polynomials of rooted digraphs factorise. We showed that there exists an infinite family of non-separable digraphs where their greedoid polynomials GM-factorise. We also characterised the greedoid polynomials of rooted digraphs that belong to the family.

We now suggest some problems for further research.
\begin{enumerate}
	\item Investigate the factorability of greedoid polynomials of rooted graphs, or even greedoids in general.	
\end{enumerate}

Gordon and McMahon gave a graph-theoretic interpretation for the highest power of $ 1 + z $ for greedoid polynomials of rooted digraphs. We could investigate a similar problem for the other basic factor $ 1 + t $.
\begin{enumerate}[resume]
	\item Does there exist a graph-theoretic interpretation for the highest power of $ 1 + t $ for greedoid polynomials of rooted digraphs?
\end{enumerate}

By Theorem~\ref{thm:infinite_family_any_graph}, we can see that there exist (totally) $ k $-nonbasic rooted digraphs for $ k \ge 3 $.
\begin{enumerate}[resume]
	\item For $ k \ge 2 $, does there exist a (totally) $ k $-primary rooted digraph?
\end{enumerate}

For rooted digraphs of order six, there are 15 totally $ 2 $-nonbasic digraphs and nine totally $ 1 $-primary digraphs.
\begin{enumerate}[resume]
	\item For $ k \ge 1 $, can we characterise greedoid polynomials of totally $ (k+1) $-nonbasic digraphs and totally $ k $-primary digraphs?
	\item Determine necessary and sufficient conditions for greedoid polynomials of rooted digraphs to factorise.
\end{enumerate}

\section*{Acknowledgement}

We thank Gary Gordon for his useful feedback.


\begin{thebibliography}{99}
	
	\bibitem{BjornerZ1992}
	A. Bj{\"o}rner and G. M. Ziegler,
	Introduction to greedoids,
	\textit{Matroid applications, Encyclopedia Math. Appl.}, Cambridge University Press, Cambridge, \textbf{40} (1992), 284--357.
	
	\bibitem{BollobasPR2000}
	B. Bollob\'{a}s and L. Pebody and O. Riordan,
	Contraction-deletion invariants for graphs,
	\textit{J. Combin. Theory Ser. B}, \textbf{80} (2000), 320--345.
	
	\bibitem{BrylawskiO1992}
	T. H. Brylawski and J. Oxley,
	The Tutte polynomial and its applications,
	\textit{Matroid applications, Encyclopedia Math. Appl.}, Cambridge University Press, Cambridge, \textbf{40}, (1992), 123--225.
	
	\bibitem{Chow2018}
	T. Y. Chow,
	Digraph analogues of the Tutte polynomial,
	in: J. Ellis-Monaghan and I. Moffatt (eds.), \textit{Handbook on the Tutte polynomial and related topics}, CRC Press, to appear.
	
	\bibitem{ChungGraham1995}
	F. R. K. Chung and R. L. Graham,
	On the cover polynomial of a digraph,
	\textit{J. Combin. Theory Ser. B}, \textbf{65} (1995), 273--290.
	
	\bibitem{EisenstatGordon2006}
	D. Eisenstat and G. P. Gordon,
	Non-isomorphic caterpillars with identical subtree data,
	\textit{Discrete Math.}, \textbf{306} (2006), 827--830.
	
	\bibitem {Gordon2008}
	G. P. Gordon,
	Chromatic and Tutte polynomials for graphs, rooted graphs, and trees,
	\textit{Graph Theory Notes N. Y.}, \textbf{54} (2008), 34--45.
	
	\bibitem {GordonMcMahon1989}
	G. P. Gordon and E. W. McMahon,
	A greedoid polynomial which distinguishes rooted arborescences,
	\textit{Proc. Amer. Math. Soc.}, \textbf{107} (2) (1989), 287--298.
	
	\bibitem {GordonMcMahon1997}
	G. P. Gordon and E. W. McMahon,
	Interval partitions and activities for the greedoid Tutte polynomial,
	\textit{Advances in Applied Math.}, \textbf{18} (1997), 33--49.
	
	\bibitem{NetworkX}
	A. A. Hagberg, D. A. Schult and P. J. Swart,
	Exploring network structure, dynamics, and function using NetworkX, in: \textit{Proceedings of the $ 7^{th} $ Python in Science Conference (SciPy2008)}, G\"{a}el Varoquaux, Travis Vaught, and Jarrod Millman (Eds), (Pasadena, CA USA), pp. 11-–15, Aug 2008.
	
	\bibitem{KorteLovasz1981}
	B. Korte and L. Lov{\'a}sz,
	Mathematical structures underlying greedy algorithms,
	\textit{Fundamentals of Computation Theory (Szeged, August 24--28, 1981)}, Lecture Notes in Comput. Sci., Springer, Berlin-New York, \textbf{117} (1981), 205--209.
	
	\bibitem{KorteLovasz1984}
	B. Korte and L. Lov{\'a}sz,
	Greedoids - A structural framework for the greedy algorithm,
	in: \textit{Progress in Combinatorial Optimization}, (1984), 221--243.
	
	\bibitem{KorteLovasz1985}
	B. Korte and L. Lov{\'a}sz,
	Polymatroid greedoids,
	\textit{J. Combin. Theory Ser. B}, \textbf{38} (1985), 41--72.
	
	\bibitem{KorteLS1991}
	B. Korte, L. Lov{\'a}sz and R. Schrader,
	\textit{Greedoids},
	Springer, Berlin, 1991.
	
	\bibitem{McMahon1993}
	E. W. McMahon,
	On the greedoid polynomial for rooted graphs and rooted digraphs,
	\textit{J. Graph Theory}, \textbf{17} (3) (1993), 433--442.
	
	\bibitem{MerinodeMierNoy2001}
	C. Merino, A. de Mier and M. Noy,
	Irreducibility of the Tutte polynomial of a connected matroid,
	\textit{J. Combin. Theory Ser. B}, \textbf{83} (2) (2001), 298--304.
	
	\bibitem{MorganF2009}
	K. J. Morgan and G. E. Farr,
	Certificates of factorisation of chromatic polynomials,
	\textit{Electron. J. Combin.}, \textbf{16} (2009), \#R74.
	
	\bibitem{Tedford2009}
	S. J. Tedford,
	A Tutte polynomial which distinguishes rooted unicyclic graphs,
	\textit{European J. Combin.}, \textbf{30} (2009), 555--569.
		
	\bibitem{Tutte1954}
	W. T. Tutte,
	A contribution to the theory of chromatic polynomials,
	\textit{Can. J. Math.}, \textbf{6} (1954), 80--91.
	
	\bibitem{Tutte1973}
	W. T. Tutte,
	Duality and trinity,
	in: \textit{Infinite and Finite Sets (Colloq., Keszthely, 1973)}, Vol. III, Colloq. Math. Soc. Janos Bolyai, Vol. 10, North-Holland, Amsterdam, 1973, pp. 1459--1475.
	
	\bibitem{Whitney1932}
	H. Whitney,
	A logical expansion in mathematics,
	\textit{Bull. Amer. Math. Soc.}, \textbf{38} (1932), 572--579.
	
	\bibitem {YowPHD2019}
	K. S. Yow,
	\textit{\href{https://figshare.com/articles/Tutte-Whitney_Polynomials_for_Directed_Graphs_and_Maps/7610882}{Tutte-Whitney Polynomials for Directed Graphs and Maps}},
	PhD Thesis, Monash University, 2019.
	
	\bibitem{YowFM2018pp_greedoid}
	K. S. Yow, K. J. Morgan and G. E. Farr,
	Factorisation of greedoid polynomials of rooted digraphs,
	\textit{Preprint}, (2018), \href{https://arxiv.org/abs/1809.02924}{https://arxiv.org/abs/1809.02924}.
	
\end{thebibliography}
\end{document}